\documentclass[12pt]{article}
\usepackage{xcolor}
\usepackage[margin=1in]{geometry} 
\usepackage{amsmath,amsthm,amssymb}
\usepackage{cancel}
\usepackage[english]{babel}
\usepackage[utf8x]{inputenc}
\usepackage{amsmath}
\usepackage{amssymb}
\usepackage{amsthm}
\usepackage{graphicx}

\usepackage{enumitem}
\usepackage{listings}
\usepackage{geometry}
\newcommand{\R}{\mathbb{R}}
\newcommand{\Z}{\mathbb{Z}}

\newcommand{\T}{\mathbb{T}}
\newcommand{\cC}{\mathcal{G}}
\newcommand{\cU}{\mathcal{U}}
\newcommand{\cV}{\mathcal{V}}
\newcommand{\Dt}{{\Delta t}}

\newcommand{\Da}{D_{\Dt}^{\alpha}}
\newcommand{\ra}{\rho_\alpha}
\newcommand{\ovu}{\overline{u}}
\newcommand{\unu}{\underline{u}}
\newcommand{\pdc}{\partial_{t}^{\alpha}}
\newcommand{\pdh}{\partial_{t}^{\frac 1 2}}
\theoremstyle{definition}
\newtheorem{theorem}{Theorem}
\numberwithin{theorem}{section}
\newtheorem{lemma}[theorem]{Lemma}
\newtheorem{corollary}[theorem]{Corollary}
\newtheorem{proposition}[theorem]{Proposition}

\newtheorem{definition}[theorem]{Definition}

\newtheorem{remark}[theorem]{Remark}

\numberwithin{equation}{section}
\begin{document}
 
\renewcommand{\thefootnote}{\fnsymbol{footnote}}	
\title{Approximation  of  Hamilton-Jacobi equations with Caputo time-fractional derivative }
\author{Fabio Camilli\footnotemark[1] \and Serikbolsyn Duisembay\footnotemark[2]}
\date{\today}
\footnotetext[1]{Dip. di Scienze di Base e Applicate per l'Ingegneria,  ``Sapienza'' Universit{\`a}  di Roma, via Scarpa 16,	00161 Roma, Italy, ({\tt e-mail: fabio.camilli@sbai.uniroma1.it})}
\footnotetext[2]{ King Abdullah University of Science and Technology
	(KAUST),  CEMSE division, Thuwal 23955-6900, Saudi Arabia, ({\tt e-mail: serikbolsyn.duisembay@kaust.edu.sa})}
\footnotetext{S. Duisembay was partially supported by KAUST baseline funds and KAUST OSR-CRG2017-3452.}	
\maketitle
\begin{abstract}
In this paper, we investigate the numerical approximation of Hamilton-Jacobi equations with the Caputo time-fractional derivative. We introduce an explicit in time discretization of the Caputo derivative   and  a   finite difference scheme for the  approximation of the Hamiltonian. We show that the  approximation scheme so obtained  is stable under an appropriate   condition on the discretization parameters and converges to the unique viscosity solution of the Hamilton-Jacobi equation. 
\end{abstract}
%%%%%%%%%%%%%%%%%%
 \begin{description}
	\item [{\bf AMS subject classification}:] 35R11,   65L12, 49L25.
	\item[{\bf Keywords}:]   Fractional Hamilton-Jacobi  equation, Caputo time derivative,   finite difference, convergence.
\end{description}

%%%%%%%%%%%%%%%%%%%%%
%                   %
%%%%%%%%%%%%%%%%%%%%%
\section{Introduction}
We define a class of finite difference  schemes for the time-fractional  Hamilton-Jacobi equation 
\begin{equation}\label{eq:HJ_caputo}
 \partial_t^\alpha u (t,x) + H(t,x,u(t,x),Du(t,x)) = 0  \qquad (t,x) \in  Q_T:=(0,T] \times \T^d,    
\end{equation}
 where $\T^d$ is unit torus in $\R^d$. The symbol  $\partial_t^\alpha$, for $0 < \alpha \le 1$, denotes  the Caputo derivative 
\begin{equation*}
\partial_t^\alpha u(t,x) = \frac{1}{\Gamma(1-\alpha)}\int_{0}^{t}\frac{\partial_s u(s,x)}{(t-s)^{\alpha}}ds  
\end{equation*}
(note that $\partial_t^\alpha$ reduces  to the standard time derivative $\partial_t$ for $\alpha=1$) and $Du$   the gradient with respect to $x$. Equation \eqref{eq:HJ_caputo} is completed with the initial condition
\begin{equation}\label{IC}
u(0,x)=u_0(x),\qquad x\in\T^d.
\end{equation}
In recent times,  there has been an increasing interest in the study of differential equations with time-fractional derivatives. Indeed, this kind of differential operators allows us to introduce new  phenomena in differential models   such as memory and trapping effects  \cite{m,ms,po,t}. Also, the numerical approximation of differential equations with fractional time-derivative has been extensively analyzed  \cite{bdst,lx,lmz}.\par
Since in general smooth solutions to  Hamilton-Jacobi equations are not expected to exist, for equation \eqref{eq:HJ_caputo}  a theory of  weak solutions, in viscosity sense,   has been introduced,  in \cite{gn,n,ty} (see also \cite{su} for a different, but closely related, notion of weak solution).
Most of the results and techniques which hold in the  classical case, i.e. for  $\alpha=1$,
have been extended to the fractional case in order to prove the well-posedness of the Hamilton-Jacobi equation \eqref{eq:HJ_caputo}. \par
In the classical case, one of the most important properties of the  viscosity solution theory   is the stability with respect to the uniform convergence (see \cite{bcd}). Starting with the seminal paper \cite{cl}, this property has generated   an enormous literature concerning the numerical approximation of Hamilton-Jacobi equations (see for example \cite{ff,o,se} and reference therein). Stability with respect to the  uniform convergence is inherited  by viscosity solutions of the Hamilton-Jacobi equation \eqref{eq:HJ_caputo}. Following \cite{cl}, we define a  general class of   finite difference    schemes for \eqref{eq:HJ_caputo}. We show that, under an appropriate Courant–Friedrichs–Lewy (CFL)   condition of the type $\Delta t^\alpha=O(\Delta x)$, these schemes are monotone, stable and consistent. Moreover, relying on an adaptation of the classical Barles-Souganidis convergence Theorem  \cite{bs}, we prove that  the numerical solutions generated  by these schemes converge  to the unique viscosity solution of the limit problem. In order  to verify the properties of the proposed schemes,  we perform several numerical tests, and to analyze the order of the approximation error,  we also compute exact solutions for some time-fractional Hamilton-Jacobi equations. 
We  have only recently become aware that a similar problem  was considered in \cite{glm}.\par
The rest of the paper is organized as follows. In Section \ref{s2}, we shortly review some basic properties of the theory of viscosity solution for \eqref{eq:HJ_caputo}. Section \ref{s3} is devoted to the description of a class of finite difference schemes and their properties. In Section \ref{s4}, we prove  a convergence result   and in Section \ref{s5} we carry out some numerical tests.

%%%%%%%%%%%%%%%%%%%%%
%                   %
%%%%%%%%%%%%%%%%%%%%%
\section{Viscosity solutions for Hamilton-Jacobi equation with  time-fractional derivative}\label{s2}
In this section, we briefly review  definitions and some results for the continuous problem \eqref{eq:HJ_caputo}
(we refer to \cite{gn,n} for more details). \\
Throughout the paper, a function $u$ on $\T^d$ will be   equivalently regarded as a function defined on $\R^d$ which is $\mathbb{Z}^d$ periodic.
We consider the following assumptions on the Hamiltonian  $H$ and on the initial datum $u_0$.
\begin{enumerate}
	\item[(H1)] $H:\overline{Q_T} \times \R \times \R^d\to\R$ is continuous;
	\item[(H2)] there exists a modulus $\omega : [0,\infty) \to [0,\infty)$ such that
	\[
	|H(t,x,r,p) - H(t,y,r,p)| \leq \omega(|x-y|(1+|p|))
	\]
	for all $(t,x,r,p), \,(t,y,r,p) \in [0,T] \times \T^d \times \R \times \R^d$;
	\item[(H3)] $r \mapsto H(t,x,r,p)$ is nondecreasing for all $(t,x,p) \in {\overline Q_T}\times\R^d$;
	\item[(H4)] $u_0:\T^d\to\R$ is a continuous function.
\end{enumerate}
For a function $f : [0,T] \to \R$ such that $f  \in C^1((0,T]) \cap C([0,T])$ and $f' \in L^1((0,T))$, the Caputo time fractional derivative is defined by
\begin{equation}\label{eq:caputo_der}
\partial_t^\alpha f(t) = \frac{1}{\Gamma(1-\alpha)}\int_{0}^{t}\frac{f'(s)}{(t-s)^{\alpha}}ds, 
\end{equation}
for any $t\in (0,T]$. Using integration by parts and change of variables,   \eqref{eq:caputo_der} can be rewritten as
\begin{equation}\label{eq:caputo_bis}
\pdc f(t) = J[f](t)+K_{(0,t)}[f](t),
\end{equation}
where
\begin{align*}
&J[f](t) := \frac{f(t)-f(0)}{t^\alpha \Gamma(1-\alpha)},\\
&K_{(0,t)}[f](t) :=  \frac{\alpha}{\Gamma(1-\alpha)} \int_{0}^{t} \frac{f(t)-f(t-\tau)}{\tau^{\alpha+1}} d \tau.
\end{align*}  
By natural extension, we also define  $K_{(a,b)}[f](t)$ for any $a,b$ with $0\le a<b\le t$.
The advantage of rewriting the Caputo derivative in the form \eqref{eq:caputo_bis} is explained in \cite{acv,gn,ty}. \\
For a set $A\subset \T^d$ and a function $u:A\to\R$, we denote by $u^*$ and $u_*$ the upper and the lower semi-continuous envelopes of $u$, i.e.
\[
u^*(x)=\lim_{r\to 0}\sup\{u(y): y\in A\cap \overline{B(x,r)}\}
\]
and $u_*(x)=-(-u)^*$, where $B(x,r)$ is a the open  ball of centre $x$ and radius $r$. We also denote by $USC(\overline{Q_T})$ (resp., $ LSC(\overline{Q_T}) $)  the class of the upper semi-continuous 
(resp., lower semi-continuous) functions in $\overline{Q_T}$.\\ 
We give the definition of viscosity solution for 
\eqref{eq:HJ_caputo} (see \cite[Definiton 2.2]{n}).
%%%%%%%%%%%%%%%%%%

	\definition
\begin{itemize}	Let $O\subset \T^d$. Then
\item [(i)] A function $u :[0,T]\times O\to \R$  is said  a viscosity subsolution  of \eqref{eq:HJ_caputo}  in $(a,T]\times O$  if 
$u^*<+\infty$  in $(a,T]\times O$ and 
\[
J[\varphi](\hat{t},\hat{x})+K_{(0,\hat t)}[\varphi](\hat{t},\hat{x}) + H(\hat{t},\hat{x},u^*(\hat{t},\hat{x}),D\varphi(\hat{t},\hat{x})) \leq 0,
\]
whenever $(\hat{t},\hat{x}) \in (a,T] \times O$
and $\varphi\in C^{1,1}((a,T]\times O)\cap C([0,T]\times O)$ satisfy
	\[
	\max_{[0,T] \times O}(u^*-\varphi) = (u^*-\varphi)(\hat{t},\hat{x}). \]
\item [(ii)] A function $u :[0,T]\times O\to \R$  is said  a viscosity supersolution  of \eqref{eq:HJ_caputo}  in $(a,T]\times O$  if 
$u_*>-\infty$  in $(a,T]\times O$ and 
\[
J[\varphi](\hat{t},\hat{x})+K_{(0,\hat t)}[\varphi](\hat{t},\hat{x}) + H(\hat{t},\hat{x},u^*(\hat{t},\hat{x}),D\varphi(\hat{t},\hat{x})) \geq 0,
\]
whenever $(\hat{t},\hat{x}) \in (a,T] \times O$
and $\varphi\in C^{1,1}((a,T]\times O)\cap C([0,T]\times O)$ satisfy
\[
\min_{[0,T] \times O}(u_*-\varphi) = (u_*-\varphi)(\hat{t},\hat{x}).
 \]
\item[(iii)] A function $u :[0,T]\times O\to \R$  is said  a viscosity solution  of \eqref{eq:HJ_caputo} in $(a,T]\times O$ if it is both a viscosity sub- and supersolution of   \eqref{eq:HJ_caputo} in $(a,T]\times O$.
\end{itemize}
 \enddefinition
 %%%%%%%%%%%
 In the previous definition, the notation $C^{1,1}((a,T]\times O)$ 
 denotes the space of functions $\varphi$ such that $\varphi_t$, $\partial_t\varphi$ and $D\varphi$ are continuous in $(a,T]\times O$.

%%%%
For    other equivalent definitions of viscosity solutions for \eqref{eq:HJ_caputo},  we refer to  \cite{gn}. \\
The first result is  a comparison principle  for \eqref{eq:HJ_caputo} (see \cite[Theorem 3.1]{gn}).
\theorem \label{thm:comp}
Assume (H1)-(H3). Let $u \in USC(\overline{Q_T})$ and $v \in LSC(\overline{Q_T})$ be a subsolution and a supersolution of \eqref{eq:HJ_caputo}, respectively. If $u(0,x) \leq v(0,x)$  for $x\in\T^d$, then $u \leq v$ on $\overline{Q_T}$.
\endtheorem
%The proof of the previous result is based on an adaptation of the classical doubling of variables method in viscosity solution theory.\\
% 
We also recall an existence result for  viscosity solutions of \eqref{eq:HJ_caputo} (see \cite[Theorem 4.2]{gn}).
%For a locally bounded function  $u: \overline{Q_T} \to \R$,  $u^*$  and $u_*$ denote respectively the upper and lower semi-continuous envelope, defined for $(t,x)\in \overline{Q_T}$ by 
%\[ u^*(t,x) = \lim_{\delta\to 0^+}\sup\{u(s,y)\,|\, (s,y)\in B((t,x),\delta)\cap \overline{Q_T}\},\] 
%and   by $u_*(x) = -(-u)^*$.
\theorem \label{thm:exist}
Assume (H1). Let $u^- \in USC(\overline{Q_T})$ and $u^+ \in LSC(\overline{Q_T})$ be a subsolution and a supersolution of \eqref{eq:HJ_caputo} such that $(u^-)_* > - \infty$ and $(u^+)^*<+\infty$ on $\overline{Q_T}$. If $u^- \leq u^+$, then there exists a solution $u$ of \eqref{eq:HJ_caputo} that satisfies $u^- \leq u \leq u^+$ in $\overline{Q_T}$.
\endtheorem
By Theorems \ref{thm:comp} and \ref{thm:exist}, it follows an existence and uniqueness result for the solution of \eqref{eq:HJ_caputo}, \eqref{IC} (see \cite[Corollary 4.3]{gn}).
\begin{corollary}
	Assume (H1)-(H4). Then there exists a unique   viscosity solution of \eqref{eq:HJ_caputo} which satisfies the initial condition \eqref{IC}.
\end{corollary}
Existence and uniqueness results for the problem  of \eqref{eq:HJ_caputo}, \eqref{IC} in a bounded domain  with boundary conditions in viscosity sense  are discussed in \cite{n}.
%%%%%%%%%%%%%%%%%%%%%%%%%%%%%%%%%%%
%                                 %
%%%%%%%%%%%%%%%%%%%%%%%%%%%%%%%%%%%
\section{A class of finite difference schemes}\label{s3}
In this section we describe a finite difference scheme for the approximation of \eqref{eq:HJ_caputo}. 
For simplicity of notations, we  assume that the Hamiltonian $H$ depends only on the state and gradient variables,
i.e. $H=H(x,p)$, and that the dimension $d$ is equal to $2$. The extension  for general $H$ and $d$  will be clear from this special case. Moreover, we will always identify a function $\T^d$ with its $\mathbb{Z}^d$-periodic extension defined in all $\R^d$.
\\ Let $\T^2_h$ be a uniform grid on the torus with step $h$, (this supposes that $1/{h}$ is an integer), and denote by $x_{i,j}$   a generic point in  $\T^2_h$ (an anisotropic mesh with steps $h_1$ and $h_2$ is possible too and we have taken $h_1=h_2$ only for simplicity). The value  $U^n_{i,j}$ denotes the numerical approximation of the function $u$ at $(x_{i,j},t_n)=(ih,jh, n\Dt)$, $0\le i,j\le 1/h$, $n=0,\dots,N$ 
(assuming that $N=T/\Dt$ is an integer). We also denote by  $U^n$  the grid function taking the value $U^n_{i,j}$ at $x_{i,j}\in \T^2_h$.\\
We start by  describing the numerical approximation   of the Caputo time-fractional derivative $\partial_t^\alpha$  introduced in \cite{lx}. 
The numerical derivative is obtained by approximating the time-derivative  inside the fractional integral
in \eqref{eq:caputo_der} via finite difference and writing in compact form the expression so obtained. We approximate $\partial_t^\alpha u(x_{i,j},t_{n+1})$ by
\begin{equation}\label{eq:der1}
\begin{aligned}
\Da U^{n+1}_{i,j}  
& = \frac{1}{\Gamma(1-\alpha)} \sum_{m=0}^{n} \int_{t_{m}}^{t_{m+1}} \frac{U_{i,j}^{m+1}-U_{i,j}^m}{\Delta t}\frac{1}{(t_{n+1}-s)^\alpha} ds \\ 
& = \frac{1}{\Gamma(1-\alpha)(1-\alpha)} \sum_{m=0}^{n}\frac{U_{i,j}^{m+1}-U_{i,j}^m}{\Delta t}\left(  - \frac{1}{(t_{n+1}-t_{m+1})^{\alpha-1}}+\frac{1}{(t_{n+1}-t_{m})^{\alpha-1}}\right)  \\ 
& = \frac{1}{\Gamma(2-\alpha)} \sum_{m=0}^{n} \frac{(n+1-m)^{1-\alpha}-(n-m)^{1-\alpha}}{\Delta t^\alpha}(U_{i,j}^{m+1}-U_{i,j}^m),
\end{aligned}
\end{equation}
since $t_{n}-t_{m}=(n-m)\Dt$.
Defined 
\begin{equation}\label{eq:rho}
\ra = \Gamma(2-\alpha) \Dt^{\alpha},
\end{equation}
we obtain by \eqref{eq:der1}
\[
\begin{aligned}
\ra \Da U^{n+1}_{i,j}= & \sum_{m=0}^{n} \Big((n+1-m)^{1-\alpha}-(n-m)^{1-\alpha}\Big)(U_{i,j}^{m+1}-U_{i,j}^m) \\ 
& = - \Big((n+1)^{1-\alpha}-n^{1-\alpha}\Big) U_{i,j}^0 \\
& - \sum_{m=1}^{n} \Big( 2(n+1-m)^{1-\alpha}-(n+2-m)^{1-\alpha} - (n-m)^{1-\alpha} \Big) U_{i,j}^m + U_{i,j}^{n+1} \\
& = U_{i,j}^{n+1} - \sum_{m=0}^{n} c_m^{n+1} U_{i,j}^m ,
\end{aligned}
\]
where 
\[
\begin{aligned}
c_0^{n+1} &= (n+1)^{1-\alpha}-n^{1-\alpha}\\
c_m^{n+1} &= 2(n+1-m)^{1-\alpha}-(n+2-m)^{1-\alpha} - (n-m)^{1-\alpha}
\end{aligned}
 \] 
for $1 \leq m \leq n$.
Thus, the approximation of the Caputo time-derivative is given by 
\begin{equation}\label{approx_Caputo}
\Da U^{n+1}_{i,j} = \frac{1}{\ra}\left( U_{i,j}^{n+1} - \sum_{m=0}^{n} c_m^{n+1} U_{i,j}^m\right) .
\end{equation}
%%%%
\begin{remark}
Denoted by $r_{\Dt}^{n+1}$ the truncation error, in \cite{lx} it is proved that 
\[ r_{\Dt}^{n+1}\le c_u \Dt^{2-\alpha} \]
where $c_u$ is a constant depending on the second order time-derivative of $u$. Hence the temporal accuracy of the scheme is
of order $2-\alpha$.
\end{remark}
In the following we summarize some properties of the coefficients  $c_m$ in \eqref{approx_Caputo}
\begin{lemma}\label{lem:c_properties}
\begin{enumerate}
\item[(i)] $c_m^{n+1}>0$ for $0\leq m \leq n$.
\item[(ii)] $c_{0}^{n+2}-c_0^{n+1}=-c_1^{n+2}$.
\item[(iii)] $c_{m+1}^{n+2}=c_{m}^{n+1} \text{ for } 1 \leq m \leq n$.
\item[(iv)] $\sum_{m=0}^{n} c_m^{n+1}=1$.
%\item[(v)] $\sum_{m=0}^n c_m^{n+1} m = (n+1)-(n+1)^{1-\alpha}$.
\end{enumerate}
\end{lemma}
\begin{proof}
\begin{enumerate}
\item[(i)] 
When $m=0$, it is clear that $c_0^{n+1}>0$. Consider the case where $1\leq m \leq n$. Because of the strong concavity of the function $x^{1-\alpha}$ for $x \geq 0$, by Jensen's inequality, we have
\[
\frac{(n+2-m)^{1-\alpha} + (n-m)^{1-\alpha}}{2} < (n+1-m)^{1-\alpha}.
\]
Thus, it follows that $c_m^{n+1}>0$.
\item[(ii)]
By definition,
\[
c_{0}^{n+2}-c_0^{n+1} = (n+2)^{1-\alpha}- 2 (n+1)^{1-\alpha}+n^{1-\alpha} = -c_1^{n+2}.
\]
\item[(iii)]
By definition,
\[
\begin{aligned}
c_{m+1}^{n+2} & = 2((n+2)-(m+1))^{1-\alpha}-((n+3)-(m+1))^{1-\alpha} - ((n+1)-(m+1))^{1-\alpha} \\
& = 2(n+1-m)^{1-\alpha}-(n+2-m)^{1-\alpha} - (n-m)^{1-\alpha} = c_m^{n+1}.
\end{aligned}
\]
\item[(iv)] We have
\[
\begin{aligned}
\sum_{m=0}^{n}& c_m^{n+1}  = (n+1)^{1-\alpha}-n^{1-\alpha} +  \sum_{m=1}^{n} \left(2 (n+1-m)^{1-\alpha} - (n+2-m)^{1-\alpha}- (n-m)^{1-\alpha} \right) \\
&=(n+1)^{1-\alpha}-n^{1-\alpha} +  \sum_{m=1}^{n} 2 (n+1-m)^{1-\alpha} - \sum_{m=0}^{n-1} (n+1-m)^{1-\alpha}- \sum_{m=2}^{n+1}(n+1-m)^{1-\alpha} \\
& = (n+1)^{1-\alpha}-n^{1-\alpha} + 2n^{1-\alpha} + 2 - (n+1)^{1-\alpha} - n^{1-\alpha} - 1 = 1.
\end{aligned}
\]
%\item[(v)] Manipulating the sums and collecting similar terms, we obtain
%\[
%\begin{aligned}
%\sum_{m=0}^n c_m^{n+1} m & = \sum_{m=1}^n 2(n+1-m)^{1-\alpha}m - (n+2-m)^{1-\alpha}m - (n-m)^{1-\alpha} m \\
%& = \sum_{m=0}^{n-1} 2(n-m)^{1-\alpha}(m+1) -  \sum_{m=-1}^{n-2} (n-m)^{1-\alpha}(m+2) - \sum_{m=1}^n (n-m)^{1-\alpha} m \\
%& = 2n^{1-\alpha} + 2 n - (n+1)^{1-\alpha} - 2 n^{1-\alpha} - (n-1) = (n+1)-(n+1)^{1-\alpha}.
%\end{aligned}
%\]
\end{enumerate}
\end{proof} 
%%%%%%%%%%%%%%%%%%  Spatial part %%%%%%%%%%%%%%%%%%%%%%%%%%%%%%%%%%%%%%%
For the approximation of  the Hamiltonian in \eqref{eq:HJ_caputo} we follow the approach in \cite{cl}.
We introduce the finite difference operators
\begin{equation} 
\label{eq:finite_diff}
(D_1^+ U )_{i,j} = \frac{ U_{i+1,j}-U_{i,j}   } {h} \quad \hbox{and }\quad  (D_2^+ U )_{i,j} = \frac{ U_{i,j+1}-U_{i,j}   } {h},
\end{equation}
and define
\begin{align}
[D_h U]_{i,j} =\left((D_1^+ U )_{i,j} , (D_1^+ U )_{i-1,j}, (D_2^+ U )_{i,j}, (D_2^+ U )_{i,j-1}\right) ^T.\label{eq:discrete_gradient}
\end{align}
In order to approximate the Hamiltonian $H$ in equation \eqref{eq:HJ_caputo}, we consider a  numerical Hamiltonian $g: \T^2 \times \R^4\to \R$,  $(x,q_1,q_2,q_3,q_4)\mapsto g\left(x,q_1,q_2,q_3,q_4\right)$
satisfying the following conditions:
\begin{itemize}
	\item[(G1)] $g$ is nonincreasing with respect to $q_1$ and $q_3$, and nondecreasing with respect to $q_2$ and $q_4$.
	\item[(G2)] $g$ is consistent with the Hamiltonian $H$, i.e. 
	\begin{displaymath}
	g(x,q_1,q_1,q_2,q_2)=H(x,q), \quad \forall x\in \T^2, \forall q=(q_1,q_2)\in \R^2.
	\end{displaymath}
	\item[(G3)]  $g$ is locally Lipschitz continuous.
	\item[(G4)] There exists a constant $C$ such that
	\[
	\left| \frac{\partial g}{\partial x}(x,q_1,q_2,q_3,q_4)\right| \le C (1+|q_1|+|q_2|+|q_3|+|q_4|), \quad \forall x \in \T^2, \ \forall
	q_1,q_2,q_3,q_4 \in \R.
	\]
\end{itemize}
Hence, recalling the approximation \eqref{approx_Caputo} of the Caputo time derivative, we consider
 the explicit finite difference scheme
\begin{equation}\label{eq:HJ_discr}
\frac{1}{\rho_{\alpha}}\left( U_{i,j}^{n+1} - \sum_{m=0}^{n} c_m^{n+1} U_{i,j}^m\right) +  S(x_{i,j},h,U^{n}_{i,j} ,[U^{n}]_{i,j})=0,    
\end{equation}
for $i,j=1,\dots,1/h$, $n=0,\dots,N-1$, where $\ra$ is defined as in \eqref{eq:rho} and
\begin{equation}\label{eq:S}
S(x_{i,j},h,U^n_{i,j},[U^n]_{i,j})=  g(x_{i,j}, (D_1^+ U^n )_{i,j}, (D_1^+ U^n )_{i-1,j} , (D_2^+ U^n )_{i,j} , (D_2^+ U^n )_{i,j-1}  ).
\end{equation} 
In \eqref{eq:S}, $[U^n]_{i,j}$ represents the  set of the values of $U^n$
used to compute the scheme at $x_{i,j}$, except that the value $U^n_{i,j}$ itself,	 and $h$ is the space discretization step.
The scheme is completed with the initial condition
\begin{equation}\label{eq:IC}
 U_{i,j}^{0}= u_0(x_{i,j}).
\end{equation}
Note that $U^{n+1}$ depends on all the past history  $U^m$, $m=0,\dots,n$ of the solution.\\
For $\alpha=1$, the scheme \eqref{eq:HJ_discr}  reduces to the standard finite difference approximation 
	\begin{equation*}%\label{eq:HJ_discr_classical}
     \frac{U_{i,j}^{n+1} -U_{i,j}^{n}}{\Dt}  +  S(x_{i,j},h,U^{n}_{i,j} ,[U^{n}]_{i,j})=0
	\end{equation*}
	of the Hamilton-Jacobi equation
	\[ \partial_t u+H(x,Du)=0. \]
%%%%%%%%%%%%%%%%%%%%%%%%%%%%%%%%  da qui devo definire la mappa %%%%%%%%%%%%%%%%%%%%
\subsection{Stability properties of the scheme}
%We rewrite the scheme \eqref{eq:HJ_discr} in the iterative form
%\begin{equation}\label{eq:HJ_expl}
%U_{i,j}^{n+1} =  \sum_{m=0}^{n} c_m^{n+1} U_{i,j}^m + \ra S(x_{i,j},h,U^{n}_{i,j} ,[U^{n}]_{i,j}), \quad n=0,\dots,N-1
%\end{equation}
We set $Q^{h,\Dt}_n=\T^2_h\times \{0,\dots,n\Dt\}$ and we  denote by $\cC$ the space of the grid functions on $\T^2_h$
and by $\cC^n$, $n=0,\dots, N$, the set of the grid function on $Q^{h,\Dt}_n$, 
i.e.   
   \[ \cC^n=\left\{\cU=\{U^m\}_{m=0}^{n} \vert \, U^m\in\cC \right\}.  \]
Moreover, we set $\|U\|_\infty=\sup_{i,j}|U_{i,j}|$ for  $U=\{U_{i,j}\}_{i,j=0}^{1/h}\in \cC$, and $\|\cU\|_\infty= \sup_{m=0,\dots,n} \|U^m\|_{\infty}$ for   $\cU=\{U^m\}_{m=0}^{n}\in \cC^n$.\\
For $n\in \{0,\dots,N-1\},$, we define a map $G^n:\cC^n\to \cC$ by
\begin{equation}\label{key}
G^{n}(\cU)_{i,j}=\sum_{m=0}^{n} c_m^{n+1} U_{i,j}^m - \ra S(x_{i,j},h,U^{n}_{i,j} ,[U^{n}]_{i,j}).
\end{equation}
Hence, the scheme \eqref{eq:HJ_discr} can be rewritten in the equivalent iterative  form 
\begin{equation}\label{eq:HJ_expl}
U^{n+1}_{i,j}=G^{n}(\cU)_{i,j}, \qquad i,j=1,\dots,\frac 1 h,\,n=0,\dots,N-1
\end{equation}
%%%%%%
\begin{definition}
	We say that the scheme \eqref{eq:HJ_expl} is \emph{monotone} if, for any $n=0,\dots,N-1$, $\cU, \cV\in \cC^n$, we have  that 
	\[  U^m\le V^m,\, m=0,\dots, n,\quad  \implies \quad  G^{n}(\cU)\le G^n(\cV), \] 
where the previous inequalities are intended in the sense of the comparison of components.
\end{definition}
Since the scheme \eqref{eq:HJ_expl} is explicit, for the monotonicity, we   need  some restriction on the approximation steps $h$ and $\Dt$, as  we will discuss later on.
%%%%%%%%
\begin{proposition} \label{prop:G_properties}
Assume that the scheme \eqref{eq:HJ_discr}  is  monotone. Then, for $n=0,\dots,N-1$, we have
\begin{enumerate}
\item[(i)] $G^n(\cU+\lambda) = G^n(\cU)+\lambda$ for any $\lambda \in \R$, $\cU\in \cC^n$ (where we identify $\lambda$ both with the constant function on $\T^2_h$ and with the element of $\cC^n$ such that   $\lambda^m=\lambda$ is the constant function on $\T^2_h$ for any $m=0,\dots,n$);
\item[(ii)] $\|G^n(\cU) - G^n(\cV)\|_{\infty} \leq \|\cU-\cV\|_{\infty}$ for any $\cU, \cV\in \cC^n$;
\item[(iii)] For the constant $C \in \R$ in (G4),
\[
\|D_hG^n(\cU)\|_\infty \le 5C \|D_h\cU\|_\infty + C \ \text{for any } \cU \in \cC^n
\]
where $D_h\cU=\{D_hU^m\}_{m=0}^n$;
%\textcolor{red}{and $\|D_h\cU\|_\infty =$};
%\item $G^n(C) \subset C$;
%\item[(iv)]  for any $U\in \cC^{n+m}$
%$$\|G^{n+m}(U) - \sum_{m=0}^{n} c_m^{n+1} G^{m}(U)\|_{\infty} \le 2m\Gamma(2-\alpha) \Delta t^{\alpha} K$$
%where  $K=\sup_{m=0,\dots,n+m}\|g(x,D_h U^m)\|_\infty$;
\item[(iv)]  for any $\cU\in \cC^{n+1}$
$$
\|G^{n+1}(\cU) -G^{n}(\cU)\|_{\infty}\leq  
(1-c_0^{n+2}) \sup_{m=0,\dots,n}\|U^{m+1}-U^m\|_\infty + 2 \Gamma(2-\alpha)\Delta t^{\alpha} K,
$$
where $K=\sup_{x_{i,j} \in \T^2_h, m=0,\dots,n+1}\|g(x_{i,j},[D_h U^m]_{i,j})\|_\infty$;
\item[(v)] for any $\cU\in \cC^{n}$
 $$\|G^{n}(\cU)\|_\infty \leq \|\cU\|_{\infty} + \Gamma(2-\alpha)  \Delta t^{\alpha} \sup_{x \in \T^2}\left|H(x,0)\right|.$$
 
\end{enumerate}
\end{proposition}
\proof
\begin{enumerate}
\item[(i)]  By Lemma  \ref{lem:c_properties}, we have 
\[
\begin{aligned}
G^n(\cU+\lambda)_{i,j} & = \sum_{m=0}^{n} c_m^{n+1} (U^m+\lambda)_{i,j}-
\ra S(x_{i,j},h,U^{n}_{i,j}+\lambda ,[U^{n}+\lambda]_{i,j}) \\
& = \sum_{m=0}^{n} c_m^{n+1} U^m_{i,j} + \lambda - \ra S(x_{i,j},h,U^{n}_{i,j} ,[U^{n}]_{i,j}) = G^n(\cU)+\lambda.
\end{aligned}
\]

%%%%%%%%

\item[(ii)]
Let $\cU, \cV \in \cC^n$ and $\lambda =\|(\cU-\cV)^+\|_\infty$. We have, in the sense of the comparison of components,
\[
U^m = V^m + (U^m-V^m) \leq V^m+ \|(U^m-V^m)^+\|_\infty = V^m+ \lambda,
\]
for $m=0,\dots,n$. By monotonicity and commutativity,
\[
G^n(\cU) \leq G^n(\cV+\lambda) = G^n(\cV) + \lambda.
\]
Hence, $G^n(\cU)-G^n(\cV) \le \|(\cU-\cV)^+\|_\infty$. Similarly, we obtain another inequality $G^n(\cU)-G^n(\cV) \ge - \|(\cU-\cV)^-\|_\infty$. These two inequalities yield the result.
%%%
\item[(iii)] Let $\tau$ be a translation operator in space, that is, $\tau_l U_{i,j} = U_{i+l_1, j+l_2}$ for $l=(l_1,l_2) \in \Z^2$ for $U \in \cC$ and defined in a similar way for $\cU\in\cC^n$. Then replacing the numerical Hamiltonian $g$ with $S$ defined in \eqref{eq:S} in the assumption (G4) and from the part (ii) above, it follows that
\[
\begin{aligned}
\|D_+^1 (G^n(\cU))_{i,j}\|_\infty  & = \frac{1}{h}\left\|(\tau_{(1,0)} G^n(\cU))_{i,j} - (G^n(\cU))_{i,j} \right\|_\infty \\
&= \frac{1}{h} \Bigg\| \sum_{m=0}^{n} c_m^{n+1} U_{i+1,j}^m - S(x_{i+1,j},h,U_{i+1,j}^n,[U^n]_{i+1,j}) \\
& \quad \ \ - \sum_{m=0}^{n} c_m^{n+1} U_{i,j}^m + S(x_{i,j},h,U_{i,j}^n,[U^n]_{i,j})\Bigg\|_\infty \\
& \le \frac{1}{h}  \Bigg\| \sum_{m=0}^{n} c_m^{n+1} U_{i+1,j}^m - S(x_{i,j},h,U_{i+1,j}^n,[U^n]_{i+1,j}) \\
& \quad \ \ - \sum_{m=0}^{n} c_m^{n+1} U_{i,j}^m + S(x_{i,j},h,U_{i,j}^n,[U^n]_{i,j})\Bigg\|_\infty \\
 & \quad \ \  + \frac{1}{h} \left\| S(x_{i,j},h,U_{i+1,j}^n,[U^n]_{i+1,j}) - S(x_{i+1,j},h,U_{i+1,j}^n,[U^n]_{i+1,j})\right\|_\infty\\
& \le \frac{1}{h}\left\| (G^n(\tau_{(1,0)} \cU))_{i,j} - (G^n(\cU))_{i,j}\right\|_\infty + C(1+4\|D_h\cU\|_\infty) \\
& \le \left\Vert\frac{\tau_{(1,0)} \cU - \cU}{h}\right\Vert_\infty + C(1+4\|D_h\cU\|_\infty)= \|D_+^1\cU\|_\infty+C(1+4\|D_h\cU\|_\infty) \\
& \le 5C\|D_h\cU\|_\infty + C
\end{aligned}
\]
since $\|D_+^1\cU\|_\infty \le \|D_h\cU\|_\infty$. We have similar estimates for the other components of $D_hG^n(\cU)$. Combining all these inequalities, we obtain the desired inequality.

%Note that the previous property implies that,
%if $\|D_h U\|_\infty \leq R$, then $\|D_hG^n(U)\|_\infty \leq R$.

%%%%%%%%%%%

%\item[(iv)] We have
%\[
%\begin{aligned}
%G^{n+1}(U)_{i,j} & - \sum_{m=0}^{n} c_m^{n+1} G^{m}(U)_{i,j} \\
%& = \sum_{m=0}^{n+1} c_{m}^{n+2} U^{m}_{i,j} -\ra S(x_{i,j},h,U^{n+1}_{i,j} ,[U^{n+1}]_{i,j})  -\sum_{m=0}^{n} c_m^{n+1} U^{m+1}_{i,j} \\
%& = c_0^{n+2} U^{0}_{i,j} + c_1^{n+2} U^{1}_{i,j} + \sum_{m=1}^{n} c_{m+1}^{n+2}U^{m+1}_{i,j} - c_0^{n+1}U^{1}_{i,j} - \sum_{m=1}^{n}c_m^{n+1} U^{m+1}_{i,j} \\
%& - \ra S(x_{i,j},h,U^{n+1}_{i,j} ,[U^{n+1}]_{i,j})  \\
%& = c_0^{n+2}U^0_{i,j} - c_0^{n+2} U^1_{i,j} -\ra S(x_{i,j},h,U^{n+1}_{i,j} ,[U^{n+1}]_{i,j}).
%\end{aligned}
%\]
%Since $U^1_{i,j} = U^0_{i,j} - \ra S(x_{i,j},h,U^{0}_{i,j} ,[U^{0}]_{i,j})$, we get
%\[
%\begin{aligned}
%\|G^{n+1}(U) & - \sum_{m=0}^{n} c_m^{n+1} G^{m}(U)\|_{\infty} \\& = \|c_0^{n+2} \ra S(x_{i,j},h,U^{0}_{i,j} ,[U^{0}]_{i,j}) - \ra S(x_{i,j},h,U^{n+1}_{i,j} ,[U^{n+1}]_{i,j})\|_{\infty} < 2 \Gamma(2-\alpha) (\Delta t)^{\alpha} K,
%\end{aligned}
%\]
%where we used that $0< c_0^{n+2} < 1$. In general, using the triangle inequality and previous estimate, we obtain
%\[
%\begin{aligned}
%\|G^{n+m}(U) & - \sum_{m=0}^{n} c_m^{n+1} G^{m}(U)\|_{\infty} \\& = \sum_{l=0}^{m-1} \| G^{n+m-l}(U) - \sum_{m=0}^{n+m-l-1} c_m^{n+m-l} G^{m}(U) \|_{\infty} \\& 
%\le \sum_{l=0}^{m-1} 2 \Gamma(2-\alpha)  \Delta t^{\alpha} K = 2m\Gamma(2-\alpha) (\Delta t)^{\alpha} K.
%\end{aligned}
%\]
%%%%%%%
\item[(iv)] Using Lemma \ref{lem:c_properties}, we have
\[
\begin{aligned}
|G^{n+1}(\cU)_{i,j} & -G^{n}(\cU)_{i,j}|
= \left|\sum_{m=0}^{n}(c_m^{n+2}-c_m^{n+1})U^m_{i,j} + c_{n+1}^{n+2} U^{n+1}_{i,j}\right.\\ 
&\left. -\ra\Big(  S(x_{i,j},h,U^{n+1}_{i,j} ,[U^{n+1}]_{i,j})-  S(x_{i,j},h,U^{n}_{i,j} ,[U^{n}]_{i,j})\Big)\right| \\
& = \left|c_{0}^{n+2} U_{i,j}^{0} - c_0^{n+1} U_{i,j}^0 + \sum_{m=0}^{n} c_{m+1}^{n+2} U_{i,j}^{m+1} - \sum_{m=1}^{n} c_{m}^{n+1} U_{i,j}^{m}\right. \\
& -\left. \ra\Big(S(x_{i,j},h,U^{n+1}_{i,j} ,[U^{n+1}]_{i,j})-  S(x_{i,j},h,U^{n}_{i,j} ,[U^{n}]_{i,j})\Big) \right|
 \\ 
 & \leq \left|\sum_{m=0}^{n}c_{m+1}^{n+2}(U_{i,j}^{m+1}-U_{i,j}^{m})\right|\\
 &+ \ra  \left|S(x_{i,j},h,U^{n+1}_{i,j} ,[U^{n+1}]_{i,j})-  S(x_{i,j},h,U^{n}_{i,j} ,[U^{n}]_{i,j}) \right|
 \\
 &\leq (1-c_0^{n+2}) \sup_{m=0,\dots,n} \|U^{m+1}-U^m\|_\infty + 2 \Gamma(2-\alpha)\Delta t^{\alpha} K.
\end{aligned}
\]
%%%%%
\item[(v)] By the consistency of scheme, it follows that $G^{n}(0) =-\ra H(x_{i,j},0)$. Hence, by property (ii), we have 
\[
\|G^{n}(\cU)\|_\infty \leq \|G^{n}(\cU) -G^{n}(0)\|_{\infty} + \|G^{n}(0)\|_{\infty} \leq \|\cU\|_{\infty} + \Gamma(2-\alpha) \Dt^{\alpha} \sup_{x \in \T^2}\left|H(x,0)\right|.
\]
\end{enumerate}
\endproof
%%%%
\begin{proposition}\label{prop:stability}
Assume that \eqref{eq:HJ_discr} is monotone and let $\{U^n\}_{n=0}^N$ be a sequence generated by the scheme with the initial condition
\eqref{eq:IC}. Then
\begin{equation}\label{eq:upper}
\|U^{n}-U^0\|_\infty \le \frac{K\Gamma(2-\alpha)}{\alpha(1-\alpha)} (n\Delta t)^{\alpha} ,
\end{equation}
where $K =\sup_{x_{i,j} \in \T^2_h, m=0,\dots,n}\|g(x,[D_h U^m]_{i,j})\|_\infty$.
\end{proposition}
\begin{proof}
For $n=1$, \eqref{eq:upper} is true since  $U_{i,j}^1=U_{i,j}^0-\ra S(x_{i,j},h,U^0_{i,j},[U^0]_{i,j})$
with $S$ defined as in \eqref{eq:S}. Arguing  by induction, assume now that \eqref{eq:upper}  is true for $j\le n$. Then using the property (iv) in Lemma \ref{lem:c_properties}, we get
\begin{equation}\label{estimate1}
\begin{split}
|U^{n+1}_{i,j}-U^{0}_{i,j}|&=\left| \sum_{m=0}^{n} c_m^{n+1} (U_{i,j}^m-U^0_{i,j})-\ra  S(x_{i,j},h,U^{n}_{i,j} ,[U^{n}]_{i,j})\right|\\
&\le \left| \sum_{m=0}^{n} c_m^{n+1} (U_{i,j}^m-U^0_{i,j})\right|+ K\Gamma(2-\alpha)\Delta t^{\alpha}\\
&\le \left( \frac{1}{\alpha(1-\alpha)}\sum_{m=0}^{n}  c_m^{n+1}m^{\alpha} +1\right) K\Gamma(2-\alpha)\Delta t^{\alpha}.
\end{split}
\end{equation}
%\\&\le 2(n+1)^{\alpha}K\Gamma(2-\alpha)\Delta t^{\alpha}.
We observe that 
\[\sum_{m=0}^{n}  c_m^{n+1}m^{\alpha}=(n+1)^{\alpha}-\sum_{m=0}^n((n+1-m)^{(1-\alpha)}-(n-m)^{(1-\alpha)})((m+1)^\alpha-m^\alpha).\]
Moreover, by the inequality $(r+1)^\beta-r^\beta\ge \beta (r+1)^{\beta-1}$ for $r \ge 0$ and $\beta\in (0,1)$ (which follows by the concavity of the function $r^\beta$), we get
%%%% See Lemma 2.2 in Giga-Liu_Mitake%%%%%%%%%%%%%%%
\begin{align*}
&(n+1-m)^{1-\alpha}-(n-m)^{1-\alpha}\ge \frac{ 1-\alpha }{(n+1)^\alpha},\\
&(m+1)^\alpha-m^\alpha\ge \frac{ \alpha }{(n+1)^{1-\alpha}}.
\end{align*}
Hence 
\[\sum_{m=0}^{n}  c_m^{n+1}m^{\alpha}\le (n+1)^\alpha-\alpha(1-\alpha),\]
and replacing the previous inequality  in \eqref{estimate1}, we get estimate \eqref{eq:upper}.
\end{proof}
%%%%%%%%%%%%%%%%%%%%%%%%
%    Examples          %
%%%%%%%%%%%%%%%%%%%%%%%% 
We discuss some classical examples of approximation scheme for Hamilton-Jacobi equations adapted to the fractional case. We consider the equation
\begin{equation}\label{eq:HJ_ex}
\partial_t^\alpha u (t,x) + H(Du(t,x)) = 0 \quad \mbox{for } (t,x) \in (0,T] \times \R
\end{equation}
with periodic boundary condition.\vskip 12pt
%%%%  UPwind %%%%%%%%%%%%%%%%%%
\emph{Upwind scheme}\\
Simple upwind schemes for the equation \eqref{eq:HJ_ex} are  
\begin{equation}\label{eq:upwind+}
U_{j}^{n+1} =  \sum_{m=0}^{n} c_m^{n+1} U_{j}^m  - \ra H\left(\frac{U^n_{j+1}-U^n_j}{h}\right)
\end{equation}
if $H$ is non-increasing, or 
\begin{equation}\label{eq:upwind-}
U_{j}^{n+1} =  \sum_{m=0}^{n} c_m^{n+1} U_{j}^m  - \ra H\left(\frac{U^n_{j}-U^n_{j-1}}{h}\right)
\end{equation}
if $H$ is non-decreasing.
The numerical Hamiltonian is given by $g(q_1,q_2)=H(q_1)$, in the first case, and
by $g(q_1,q_2)=H(q_2)$ in the second case. In both cases, $g$ is monotone, consistent and regular
if $H$ is locally Lipschitz. 

Now, we establish a condition under which the scheme \eqref{eq:upwind+} is monotone. By construction, the monotonicity with respect to $U^n_{j+1}$ is obvious. Also, since by Lemma  \ref{lem:c_properties} all the coefficients $c_m^{n+1}$ are positive, the map $G^n$ is increasing with respect to the variable $U_j^m$,  $m=0,\dots, n-1$. Moreover,  \eqref{eq:upwind+} is non-decreasing with respect to $U^n_j$ if $c^{n+1}_n + \frac{\ra}{h} H'\left(\frac{U^n_{j+1}-U^n_j}{h}\right)\ge 0$. Recalling that $c_n^{n+1}=2-2^{1-\alpha}$, we get the CFL condition
 \begin{equation}\label{eq:CFL_upwind}
 \frac{\Dt^\alpha}{h}\sup_{p} |H'(p)|\le \frac{2-2^{1-\alpha}}{\Gamma(2-\alpha)}
 \end{equation}
for $p\in\R$. The same condition is necessary also for \eqref{eq:upwind-}. \vskip 12pt

%%%%%%% Lax-Freidrichs   %%%%%%
\emph{Lax-Friedrichs scheme}\\
The Lax-Friedrichs scheme is given by 
\begin{equation}\label{eq:LF_scheme}
U_j^{n+1} = \sum_{m=0}^{n} c_m^{n+1} U_j^m - \ra \left[ H\left(\frac{U_{j+1}^n - U_{j-1}^n}{2 h}\right)-\frac{(U_{j+1}^n + U_{j-1}^n-2 U_{j}^n) \theta}{\ra}\right],
\end{equation}
where $\theta$ has to  be chosen in order to satisfy the CFL condition.
Therefore, the numerical Hamiltonian  $g$ is
\[
g(q_1, q_2) = H\left(\frac{q_1 + q_2}{2}\right)-\frac{(q_1-q_2)\theta}{\lambda}
\]
for $\lambda = \rho_{\alpha}/h$ and $q_1, q_2 \in \R$. For the monotonicity of the scheme with respect to $U^n_j$, we need the condition
\[
c_n^{n+1}-2\theta \ge 0,
\]
and, for the monotonicity with respect to $U^n_{j\pm 1}$,
\[
\theta- \frac{\ra}{2 h} \sup_{p} |H'(p)| \ge 0.
\]
Then, recalling \eqref{eq:rho} the monotonicity of the scheme is implied by the CFL condition 
\begin{equation}\label{eq:CFL_LF}
\frac{\Gamma(2-\alpha) \Delta t^{\alpha}}{2h} \sup_{p} |H'(p)| \le \theta \le 1-2^{-\alpha}
\end{equation}
for $p\in\R$.
\begin{remark}
	The CFL conditions \eqref{eq:CFL_upwind} and \eqref{eq:CFL_LF} reduce to the classical ones for $\alpha=1$. In general, they become more and more restrictive for $\alpha$ decreasing to $0^+$.
	This phenomenum has been also observed in \cite{lmz} in the study of approximation schemes for
	time-fractional conservation laws. 
\end{remark}
%%%%%%%%%%%%%%%%%%%%%%%%%%%%%%%%
%       CoNVERGENCE            %
%%%%%%%%%%%%%%%%%%%%%%%%%%%%%%%%
\section{A convergence result for  the finite difference scheme}\label{s4}
In this section, we study the convergence of the scheme \eqref{eq:HJ_discr} following the classical stability argument in \cite{bs}, where it is proved that a \emph{monotone}, \emph{stable} and \emph{consistent} approximation scheme converges to the unique solution of the continuous Hamilton-Jacobi equations.\par
 We recall the definition of the relaxed limit for a  locally bounded
sequence $\{u_\sigma\}_{\sigma>0}$. The upper relaxed limit is given by
\[ (\limsup_{\sigma\to 0^+}{}^* u_\sigma)(t,x)=\lim_{\delta\to 0} \sup\left\lbrace 
u_\sigma(s,y):\,(s,y)\in Q_T\cap\overline{B_\delta(t,x)},\, 0<\sigma<\delta  \right\rbrace, \]
while the lower relaxed limit by 
$\liminf^*_{\sigma\to 0^+} u_\sigma=-\limsup^*_{\sigma\to 0^+}(-u_\sigma).$\\
We consider a sequence of approximation steps $(\Dt,h(\Dt))$ such that  $h(\Dt)\to 0$ for $\Dt\to 0$ and  we denote with $u_{\Delta t}$ the piecewise constant extension to $\overline{Q_T}$  of the solution of the approximation scheme \eqref{eq:HJ_discr} corresponding to the previous parameters, i.e.
$u_\Dt(t,x)=U^n_{i,j}$ if $n=[t/\Dt]$ (where $[\cdot]$ denotes the integer part) and $(i,j)$ such that $x\in ((i-1/2)h, (i+1/2)h]\times ((j-1/2)h, (j+1/2)h]$.

%%%%
\begin{theorem}\label{thm:conv}
Assume that   $g$ satisfies 
(G1)-(G4), $u_0$ is Lipschitz continuous and, for $\Dt$ sufficiently small and $h=h(\Dt)$, the scheme \eqref{eq:HJ_discr} is monotone.  As $\Dt\to 0^+$, the sequence $\{u_{\Delta t}\}_{{\Delta t}>0}$ converges locally uniformly to the unique   viscosity solution $u$ of \eqref{eq:HJ_caputo}, \eqref{IC}.
\end{theorem}
%%%%%%%%%%%%%%%%%%%%%%%%%%%%%%%%%%%%%%%%%
We need a preliminary result about the convergence of the fractional derivative for a test function.
\begin{lemma}\label{lem:preliminary}
Let $\varphi\in C^{1,1}((0,T]\times \T^d)\cap C([0,T]\times \T^d)$ be a test function and let  $ (t_\Dt,x_\Dt)\to (t,x)\in (0,T)\times \T^d$  for $\Dt\to 0$. Then, defined for $n_\Dt+1=[t_\Dt/\Dt]$ the discrete fractional derivative
\begin{equation}\label{sp0}
D^\alpha_\Dt \varphi ((n_\Dt+1)\Dt,x_\Dt)=\frac{1}{\rho_{\alpha}}\left( \varphi((n_\Dt+1)\Dt , x_\Dt ) - \sum_{m=0}^{n_\Dt} c_m^{n_\Dt+1} \varphi(m \Dt , x_\Dt )\right),
\end{equation}
we have
\[ \lim_{{\Delta t}\to 0} 
 D^\alpha_\Dt \varphi ((n_\Dt+1)\Dt,x_\Dt) =\partial_t^\alpha \varphi (t,x) . \]
\end{lemma}
\begin{proof}
	Convergence properties of the discrete   fractional derivative to the continuous one are already studied in \cite{jlz,lx}. Here we give a different proof of the  convergence result for  reader's convenience.
	To simplify the notation, since in this argument only the time variable is involved, we omit the dependence of $\varphi$ on $x$. Because of the continuity of the Caputo derivative of $\varphi$ with respect to $t$ (see \cite[Prop. 2.1]{n}), it is sufficient  to prove that
	\[ \lim_{{\Delta t}\to 0} 
	\left( D^\alpha_\Dt \varphi ((n_\Dt+1)\Dt)-\partial_t^\alpha \varphi (t_{{\Delta t}})\right) =0. \]
	Moreover, for a test function $\varphi$, the Caputo derivative can be defined in the standard way, see
	\eqref{eq:caputo_der}. In the rest of the proof, we omit the index ${\Delta t}$ and we write $t$, $n$ and in place of $t_{\Delta t}$, $n_{\Delta t}$. Fix $\eta>0$ such that $t>2\eta$ and let ${\bar n}<n$ be the greatest integer such that ${\bar n} \Dt\leq \eta$. Then we write
	\begin{equation} \label{disc_part_der}
	\begin{aligned}
	&D^\alpha_\Dt \varphi (t+\Delta t)-\partial_t^\alpha \varphi (t) = \frac{1}{\Gamma(1-\alpha)} \sum_{j=0}^n \int_{t_j}^{t_{j+1}} \Big(\frac{\varphi(t_{j+1})-\varphi(t_j)}{\Delta t (t-s)^{\alpha}}- \frac{\varphi'(s)}{(t-s)^{\alpha}} \Big) ds \\
	& = \frac{1}{\Gamma(1-\alpha)} \sum_{j=0}^{\bar n-1}\left( 
	\int_{t_j}^{t_{j+1}} \frac{\varphi(t_{j+1})-\varphi(t_j)}{\Delta t (t-s)^{\alpha}}ds- \int_{t_j}^{t_{j+1}} \frac{\varphi'(s)}{(t-s)^{\alpha}} ds\right)  \\
	& + \frac{1}{\Gamma(1-\alpha)} \sum_{j=\bar n}^{n}\left( 
	\int_{t_j}^{t_{j+1}} \frac{\varphi(t_{j+1})-\varphi(t_j)}{\Delta t (t-s)^{\alpha}}ds- \int_{t_j}^{t_{j+1}} \frac{\varphi'(s)}{(t-s)^{\alpha}} ds\right) .
	\end{aligned}
	\end{equation}
	We estimate the two sums (multiplied by $\Gamma(1-\alpha)$) in \eqref{disc_part_der} separately. For $0 \leq j \leq \bar n-1$, use the integration by parts to get
	\begin{align*}
	\int_{t_j}^{t_{j+1}} \frac{\varphi'(s)}{(t-s)^{\alpha}}ds & = \Big[\frac{\varphi(t_{j+1})}{(t-t_{j+1})^\alpha}-\frac{\varphi(t_{j})}{(t-t_{j})^\alpha}\Big] -\alpha\int_{t_j}^{t_{j+1}} \frac{\varphi(s)}{(t-s)^{\alpha+1}} ds \\
	& = \frac{\varphi(t_{j+1})-\varphi(t_j)}{(t-t_j)^{\alpha}} + \alpha\int_{t_j}^{t_{j+1}} \frac{\varphi(t_{j+1})-\varphi(s)}{(t-s)^{\alpha+1}} ds.
	\end{align*}
	Hence,
	\begin{equation}\label{first_term}
	\begin{aligned}
	& \int_{t_j}^{t_{j+1}} \frac{\varphi(t_{j+1})-\varphi(t_j)}{\Delta t (t-s)^{\alpha}} ds -\int_{t_j}^{t_{j+1}} \frac{\varphi'(s)}{(t-s)^{\alpha}} ds \\
	& = \frac{\varphi(t_{j+1})-\varphi(t_j)}{\Delta t} \int_{t_j}^{t_{j+1}} \Big[\frac{1}{(t-s)^{\alpha}}-\frac{1}{(t-t_j)^{\alpha}}\Big] ds - \alpha\int_{t_j}^{t_{j+1}} \frac{\varphi(t_{j+1})-\varphi(s)}{(t-s)^{\alpha+1}} ds.
	\end{aligned}
	\end{equation}
	Observe that $\frac{1}{t-s} \leq \frac{1}{t-t_{j+1}} \leq \frac{1}{t-\eta}$ for $t_j \leq s \leq t_{j+1}$.
	For the first term of \eqref{first_term}, using the inequality $t^{\alpha}-s^{\alpha} \leq (t-s)^{\alpha}$ for $0 \leq s \leq t$, we have the following estimate:	
	\begin{align*}
	&\frac{\varphi(t_{j+1})-\varphi(t_j)}{\Delta t} \int_{t_j}^{t_{j+1}} \Big[\frac{1}{(t-s)^{\alpha}}-\frac{1}{(t-t_j)^{\alpha}}\Big] ds \\
	& \leq \frac{\int_{t_{j}}^{t_{j+1}} |\varphi'(s)|ds} {\Delta t} \int_{t_{j}}^{t_{j+1}} \frac{(t-t_j)^{\alpha}-(t-s)^{\alpha}}{(t-s)^{\alpha}(t-t_j)^{\alpha}}ds \\
	& \leq  \frac{\int_{t_{j}}^{t_{j+1}} |\varphi'(s)|ds} {\Delta t} \int_{t_{j}}^{t_{j+1}} \frac{(s-t_j)^{\alpha}}{(t-s)^{\alpha}(t-t_j)^{\alpha}}ds \\
	& \leq  \frac{\int_{t_{j}}^{t_{j+1}} |\varphi'(s)|ds} {\cancel{\Delta t}} \frac{(\Delta t)^\alpha}{(t-t_{j})^\alpha (t-t_{j+1})^{\alpha}} \cancel{\Delta t} \\
	&\leq \frac{1}{(t-\eta)^{2\alpha}} (\Delta t)^\alpha \int_{t_{j}}^{t_{j+1}} |\varphi'(s)|ds.
	\end{align*}	
	For the second term of \eqref{first_term}, we have the following estimate:
	\begin{align*}
	& \alpha\int_{t_j}^{t_{j+1}} \frac{\varphi(t_{j+1})-\varphi(s)}{(t-s)^{\alpha+1}} ds \leq \alpha\frac{\omega_{\varphi}(\Delta t)}{(t-t_{j+1})^{\alpha+1}} \Delta t \leq \alpha\frac{\omega_{\varphi}(\Delta t)}{(t-\eta)^{\alpha+1}} \Delta t,
	\end{align*}
	where $\omega_{\varphi}$ is a modulus of continuity of $\varphi$.
	Thus,
	\begin{align*}
	&\Big\vert \frac{1}{\Gamma(1-\alpha)} \sum_{j=0}^{\bar n-1}
	\int_{t_j}^{t_{j+1}} \frac{\varphi(t_{j+1})-\varphi(t_j)}{\Delta t (t-s)^{\alpha}}ds - \int_{t_j}^{t_{j+1}} \frac{\varphi'(s)}{(t-s)^{\alpha}} ds \Big\vert \\
	& \leq \frac{1}{\Gamma(1-\alpha)} \sum_{j=0}^{\bar n-1}  \frac{1}{(t-\eta)^{2\alpha}} (\Delta t)^\alpha \int_{t_{j}}^{t_{j+1}} |\varphi'(s)|ds+\alpha\frac{\omega_{\varphi}(\Delta t)}{(t-\eta)^{\alpha+1}} \Delta t \\
	& \leq \frac{1}{\Gamma(1-\alpha)} \frac{1}{(t-\eta)^{2\alpha}} (\Delta t)^\alpha \int_{0}^{\eta} |\varphi'(s)|ds + \frac{\alpha}{\Gamma(1-\alpha)}\frac{\omega_{\varphi}(\Delta t)}{(t-\eta)^{\alpha+1}} \eta.
	\end{align*}
	Clearly, both terms converge to $0$ as $\Delta t \to 0$.\\
	Now, we estimate the second sum in \eqref{disc_part_der}. We have
		\begin{align*}
	&\Big \vert	\int_{t_j}^{t_{j+1}} \frac{\varphi(t_{j+1})-\varphi(t_j)}{\Delta t (t-s)^{\alpha}}ds- \int_{t_j}^{t_{j+1}} \frac{\varphi'(s)}{(t-s)^{\alpha}} ds\Big \vert\\
		& =\Big \vert \int_{t_j}^{t_{j+1}} \frac{\varphi(t_{j+1})-\varphi(t_j)-\Delta t \varphi'(s)}{\Delta t (t-s)^{\alpha}}ds \Big \vert 
		= \Big \vert \int_{t_j}^{t_{j+1}} \frac{\int_{t_j}^{t_{j+1}}\varphi'(\tau)d\tau-\Delta t \varphi'(s)}{\Delta t (t-s)^{\alpha}}ds \Big \vert \\
		=&\Big \vert \int_{t_j}^{t_{j+1}} \frac{\int_{t_j}^{t_{j+1}}[\varphi'(\tau)-\varphi'(s)]d\tau}{\Delta t (t-s)^{\alpha}}ds \Big \vert \le\omega_{\varphi'}(\Delta t)\int_{t_j}^{t_{j+1}}\frac{1}{(t-s)^{\alpha}}ds,
		\end{align*}
		where $\omega_{\varphi'}$ is the modulus of continuity of $\varphi'$ on $[\eta,t]$. Summarizing these estimates, we conclude
		\begin{align*}
		\Big \vert  \sum_{j=\bar n}^{n}\left(
		\int_{t_j}^{t_{j+1}} \frac{\varphi(t_{j+1})-\varphi(t_j)}{\Delta t (t-s)^{\alpha}}ds- \int_{t_j}^{t_{j+1}} \frac{\varphi'(s)}{(t-s)^{\alpha}} ds\right) \Big \vert 
		\le \omega_{\varphi'}(\Delta t)\int_{t_{\bar n}}^{t_{n+1}}\frac{1}{(t-s)^{\alpha}}ds\to 0,
		\end{align*} 
		as $\Delta t \to 0$ .
\end{proof}

%%%%%%%%%%%%%%%%%%%%%%%%%%
\begin{proof}[Proof of Theorem \ref{thm:conv}]
In order to apply the Barles-Souganidis' convergence result (see \cite{bs}), we define
for $(t,x)\in {\overline {Q_T}}$ 
	\begin{align*}
	&\overline{u}(t,x)=(\limsup_{{\Delta t}\to 0^+}{}^* u_{\Delta t})(t,x),\\
	&\underline{u}(t,x)=(\liminf_{{\Delta t}\to 0^+}{}^* u_{\Delta t})(t,x).
	\end{align*}
Note that, by definition,  $\unu(t,x)\le \ovu(t,x)$. 
We claim that $\ovu$, $\unu$ are, respectively, a viscosity subsolution and a viscosity supersolution of \eqref{eq:HJ_caputo} such that $\ovu(0,x)\le\unu(0,x)$ for $x\in\T^2$. If the claim holds, then
from Theorem \ref{thm:comp} it follows that $\ovu(t,x)\le \unu(t,x)$ and therefore
$u=\ovu\equiv\unu$  is the unique viscosity solution of \eqref{eq:HJ_caputo} in $Q_T$. Moreover, the
definition of $\ovu$, $\unu$ implies the uniform convergence of $\{u_{\Delta t}\}_{{\Delta t}>0}$ to $u$.\\
To prove the claim, we first observe that \eqref{eq:upper} and the continuity of $u_0$ implies that 
$\unu=\ovu=u_0(x)$ for $x\in\T^2$. Clearly, by \eqref{eq:IC}  we have $\unu\le u_0 \le \ovu$. Moreover, if  $(s_{\Delta t},y_{\Delta t})\to (0,x)$ for $\Dt\to 0$, then define
$n_{\Delta t}=[s_{\Delta t}/\Dt]$ and let $i_{\Delta t},j_{\Delta t}$ be such that $y_{\Delta t}\in ((i_{\Delta t}-1/2)h, (i_{\Delta t}+1/2)h]\times ((j_{\Delta t}-1/2)h, (j_{\Delta t}+1/2)h]$. We have
\begin{align*}
u_{\Delta t}(s_{\Delta t},y_{\Delta t})&=U^{n_{\Delta t}}_{i_{\Delta t},j_{\Delta t}}\le U^0_{i_{\Delta t},j_{\Delta t}}+2K \Gamma(2-\alpha)(n_{\Delta t}\Dt)^\alpha\\
& \le u_0(x)+L_0|x-(i_{\Delta t}h ,j_{\Delta t}h )|+2K\Gamma(2-\alpha)(n_{\Delta t}\Dt)^\alpha,
\end{align*}
where $L_0$ is the Lipschitz constant  of $u_0$ and $K=\sup\{|g(x,q)|: x\in\T^2,\, |q|\le L_0\}$.
Passing to the limit in the previous inequality  for ${\Delta t}\to 0^+$, we get $\limsup_{\Delta t\to 0} u_{\Delta t}(s_{\Delta t},y_{\Delta t})\le u_0(x)$ which implies, for the arbitrariness of the sequence $(s_{\Delta t},y_{\Delta t})$,
$\ovu(0,x)\le u_0(x)$. We prove similarly that $\unu(0,x)\ge u_0(x)$. \\
The stability of the scheme \eqref{eq:HJ_expl}, i.e.   the boundedness of the sequence $\{u_{\Delta t}\}_{{\Delta t}>0}$   bounded uniformly in ${\Delta t}$, is clearly implied by     Prop. \ref{prop:stability}.\\
To prove the consistency of the scheme, we  claim  that, given a test function $\varphi$ and a sequence $(t_{{\Delta t}},x_{\Delta t})=(n_{\Delta t}\Dt, (i_{\Delta t} h, j_{\Delta t} h))$ converging to $(t,x)\in Q_T$ for ${\Delta t}\to 0$, then we have
\begin{equation}\label{eq:consistency} 
\begin{split}
\lim_{{\Delta t}\to 0}& 
D^\alpha_\Dt \varphi (t_{\Delta t}+\Dt,x_{\Delta t})+  S\left(x_{\Delta t},h,\varphi(t_{{\Delta t}}, x_{\Delta t}) ,[\varphi(t_{{\Delta t}},\cdot) ]_{(i_{\Delta t},j_{\Delta t})}\right)\\
&=\partial_t^\alpha \varphi (t,x) + H(x,D\varphi(t,x)),
\end{split}
\end{equation}
where $D^\alpha_\Dt \varphi (t_{\Delta t}+\Delta t,x_{\Delta t})$ is defined as in \eqref{sp0}.
%\[ D^\alpha_\Dt \varphi (t_{\Delta t}+\Delta t,x_{\Delta t})=\frac{1}{{\Delta t}_{\alpha}}\left( \varphi(t_{{\Delta t}}+\Dt, x_{\Delta t}) - \sum_{m=0}^{n_{\Delta t}} c_m^{n_{\Delta t}+1} \varphi(t_{{\Delta t}}-(n_{\Delta t}-m) \Dt, x_{\Delta t})\right)\\
% \]
The previous claim follows by Lemma \ref{lem:preliminary}   and by the  equality
\[ \lim_{{\Delta t}\to 0} S\left(x_{\Delta t},h,\varphi(t_{{\Delta t}}, x_{\Delta t}) ,[\varphi(t_{\Delta t},\cdot) ]_{i_{\Delta t},j_{\Delta t}}\right)=H(x,D\varphi(t,x)),\]
which is consequence of the      assumptions (G2)-(G3) for the numerical Hamiltonian $g$.
Hence  the claim \eqref{eq:consistency} holds and we conclude that  the scheme \eqref{eq:HJ_expl}, besides   monotone, it is also stable and  consistent with the continuous equation \eqref{eq:HJ_caputo}. \\
We  prove that, by the monotonicity, stability and consistency properties of the  scheme, it follows that
$\ovu$, $\unu$ are, respectively, a viscosity subsolution and a viscosity supersolution of \eqref{eq:HJ_caputo}. We only show that $\ovu$ is a subsolution, since the argument for $\unu$ is similar. First observe that the stability of the scheme implies that the sequence $\{u_\Dt\}_{\Dt>0}$ is bounded and therefore $\ovu$ is well defined. Consider a test function $\varphi$ such that $\ovu-\varphi$ takes a strict maximum point at $(\hat t,\hat x)\in \T^d\times(0,T)$. By Lemma V.I.6 in \cite{bcd}, there exists a sequence $\{(t_{{\Delta t}}, x_{{\Delta t}})\}_{{\Delta t}>0}$ such that $u_{{\Delta t}}-\varphi$ has a maximum point at 
$(t_{{\Delta t}}, x_{{\Delta t}})$, moreover  $(t_{{\Delta t}}, x_{{\Delta t}})\to (\hat t,\hat x)
$ and $u_{\Delta t}(t_{\Delta t},x_{\Delta t})\to \ovu(\hat t, \hat x)$ for ${\Delta t}\to 0$.
Set $\delta_{\Delta t}=u_{\Delta t}(t_{\Delta t}x_{\Delta t})-\varphi(t_{\Delta t},x_{\Delta t})$ and observe that
\begin{equation}\label{sp_1}
u_{\Delta t}(t,x)\le \varphi(x,t)+\delta_{\Delta t}\qquad \forall (x,t)\in\T^d\times(0,T).
\end{equation}
Let $n_{\Delta t}$ be such that $n_{\Delta t}+1=[t_{\Delta t}/{\Delta t}]$
and  $i_\Dt,j_\Dt$ such that $x_\Dt\in   ((i_{\Delta t}-1/2)h, (i_{\Delta t}+1/2)h]\times ((j_{\Delta t}-1/2)h, (j_{\Delta t}+1/2)h]$ . Define  $\Phi^m_{i,j}=\varphi(m\Dt, x_{i,j} )$, for $i,j=1,\dots,1/h$, $m=0,\dots,[T/\Dt]$ and let $\omega_\varphi$ be a modulus of continuity for $\varphi$ . By \eqref{sp_1}, we have 
\begin{equation}\label{sp_2}
U^m_{i,j}\le \Phi^m_{i,j}+\delta_{\Delta t},   
\end{equation}
for $m=0, \dots, [T/\Dt]$, $i,j=0,\dots,h$.
Hence,   recalling that $u_{\Delta t}$ is the piecewise constant interpolation of  the solution of   \eqref{eq:HJ_discr} and the scheme can be rewritten in the equivalent form \eqref{eq:HJ_expl}, by \eqref{sp_2}, Prop. \ref{prop:G_properties}.(i) and the monotonicity of the scheme, we have 
\begin{align*}
\Phi^{n_\Dt+1}_{i_\Dt,j_\Dt} &=\varphi((n_\Dt+1)\Dt,(i_{\Delta t},j_{\Delta t}))\le \varphi(t_\Dt,x_\Dt)+\omega (\Dt+h)
\\
&=u_\Dt(t_\Dt,x_\Dt)-\delta_\Dt+\omega (\Dt+h) 
=U^{n_\Dt+1}_{i_\Dt,j_\Dt}-\delta_\Dt+\omega (\Dt+h)\\
&=G^{n_\Dt}(\cU)_{i_\Dt,j_\Dt}-\delta_\Dt+\omega (\Dt+h)\le G^{n_\Dt}(\Phi+\delta_\Dt)_{i_\Dt,j_\Dt}-\delta_\Dt+\omega (\Dt+h)\\
&= G^{n_\Dt}(\Phi)_{i_\Dt,j_\Dt}+\omega(\Dt+h).
\end{align*}
Hence
\begin{align*}
D^\alpha_\Dt &\varphi ((n_\Dt+1)\Dt,(i_{\Delta t}h,j_{\Delta t}h))\\&+    S\left((i_{\Delta t}h,j_{\Delta t}h),h,\varphi((n_\Dt+1)\Dt,(i_{\Delta t}h,j_{\Delta t}h)) ,[\varphi((n_\Dt+1)\Dt ,\cdot) ]_{i_{\Delta t},j_{\Delta t}}\right)\le\omega (\Dt+h)      
\end{align*}
and, using the consistency of the scheme and 
$((n_\Dt+1)\Dt,(i_{\Delta t}h,j_{\Delta t}h))\to(\hat t,\hat x)$ , we finally get
\[
\partial_t^\alpha \varphi (\hat t,\hat x)+H(\hat x,D\varphi(\hat t,\hat x))\le 0.
\]

%Moreover the continuous problem
%\eqref{eq:HJ_caputo} satisfies a  Comparison Principle, see Theorem \ref{thm:comp}.
%Hence, arguing as in  Theorem 2.1 in \cite{bs}, it follows that $\ovu$, $\unu$ are, respectively, a viscosity subsolution and a viscosity supersolution of \eqref{eq:HJ_caputo} and  the uniform convergence of the sequence $\{u_{\Delta t}\}_{{\Delta t}>0}$ to the unique viscosity solution of \eqref{eq:HJ_caputo}.
\end{proof}
\begin{remark}
Error estimates for the  approximation   of   Hamilton-Jacobi  equations with standard time derivative have been discussed by several authors  (see for example \cite{cdi}, \cite{cl}).  We aim to investigate this point in the future.
\end{remark}
%%%%%%%%%%%%%%%%%%%%%%%%%%%%%%%%%
%    Numerical Tests            %
%%%%%%%%%%%%%%%%%%%%%%%%%%%%%%%%%   
\section{Explicit solutions and numerical tests}\label{s5}
In this section, we implement  Lax-Friedrichs scheme  to test the convergence of the method. 
The    problems     we consider  are   in $\R^d$ and therefore the convergence   is not guaranteed by the results of the previous sections, which are in the periodic case.  Nevertheless, as the next examples show,  we obtain a correct approximation of the exact viscosity solution of the problems.
%We also compute explicit solutions of the continuous problems to have an estimate of the approximation error.
%%%%%%%%%%%%%%%%%%%%%%% Test 1 %%%%%%%%%%%%%%%%%%%%%%%%%%%%%%%%%%%
\subsection{Test 1}

First, we consider the following Hamilton-Jacobi equation:
\begin{equation}\label{eq:HJ_ex1}
\left\{
\begin{array}{ll}
\pdc u+ \frac{|Du|^2}{2}=0, \qquad &(t,x)\in (0,\infty)\times\R^d ,\\[6pt]
u_0(x)=\min\{0, |x|^2 - 1\}, &x\in\R^d.
\end{array}
\right.
\end{equation}
It is easy to verify that, if $\alpha=1$, then the unique viscosity solution of \eqref{eq:HJ_ex1} is given by
\begin{equation*}
u (t,x) = \min\left\{0, \frac{|x|^2}{1 + 2t} - 1\right\}.
\end{equation*}
We claim that a solution of \eqref{eq:HJ_ex1} for $\alpha \in (0,1)$ is given by  
\begin{equation}\label{eq:explicit_sol}
u_\alpha (t,x) = \min\left\{0, |x|^2 f(t) - 1\right\}
\end{equation}
with $f(t)$ non-negative function to be determined.
 Replacing into the equation \eqref{eq:HJ_ex1} for  $|x|\le \sqrt{1/f(t)}$ and taking into account the initial datum, we find that the function $f(t)$ has to satisfy  the fractional differential equation
\begin{equation}\label{eq:f}
\left\{
\begin{array}{l}
\pdc f(t)+ 2f(t)^2=0, \\[4pt]
f(0)=1.
\end{array}
\right.
\end{equation}
We check that \eqref{eq:explicit_sol} gives a viscosity solution of \eqref{eq:HJ_ex1}. Note that, since $u_\alpha$ is defined as the minimum of two solutions of the equation, by \cite[Lemma 4.1]{gn} it is a supersolution  in $\R^d$. Moreover, for $|x|\neq 1/\sqrt{f(t)}$, $u_\alpha$ is differentiable and the equation is satisfied in point-wise sense (see \cite[Prop. 2.10]{gn}), hence we only have to check the viscosity subsolution condition at $|x|=1/\sqrt{f(t)}$. Let $\varphi$ be a test function such that $u_\alpha-\varphi$ has a maximum point at $(t,1/\sqrt{f(t)})$ (for the other point we proceed in a similar way). Recalling  \eqref{eq:caputo_bis} and, since
\[ \varphi \big(t,\frac{1} {\sqrt{f(t)}}\big)-\varphi\big(t-\tau,\frac{1} {\sqrt{f(t)}}\big)\le u_\alpha\big(t,\frac{1} {\sqrt{f(t)}}\big)-u_\alpha\big(t-\tau,\frac{1} {\sqrt{f(t)}}\big)\quad \text{ for $\tau\le t$,}\]
 we see that $\varphi$ has to  satisfy
\begin{align*}
 \pdc \varphi\big(t,\frac{1} {\sqrt{f(t)}}\big)&=J[\varphi]\big(t,\frac{1} {\sqrt{f(t)}}\big)
+K_{(0,t)}[\varphi]\big(t,\frac{1} {\sqrt{f(t)}}\big)\le J[u_\alpha]\big(t,\frac{1} {\sqrt{f(t)}}\big)\\
&+K_{(0,t)}[u_\alpha]\big(t,\frac{1} {\sqrt{f(t)}}\big)=\pdc u_\alpha\big(t,\frac{1} {\sqrt{f(t)}}\big)=\frac{1}{f(t)}\pdc f(t).
\end{align*}
Moreover,  $0\le D\varphi(t,1/\sqrt{f(t)}) \le 2 f(t)/\sqrt{f(t)}$ and, therefore, substituting in \eqref{eq:HJ_ex1}, we get
\[\pdc \varphi \big(t,\frac{1} {\sqrt{f(t)}}\big)+
\frac 1 2 |D\varphi\big(t,\frac{1} {\sqrt{f(t)}}\big)|^2\le \frac{1}{f(t)}\pdc f(t)+ 2 f(t)= 0.  \]
Nevertheless, to compute $u_\alpha$, we need to compute the function $f$.
We look for a solution of \eqref{eq:f} in the form of a power series $f(t)=\sum_{n=0}^\infty f_nt^{\alpha n}$. Replacing in the equation  \eqref{eq:f} and observing that $\pdc t^0=0$, we have
\begin{equation}\label{eq:f1}
\sum_{n=1}^\infty f_n \pdc t^{\alpha n}+2 \sum_{n=0}^\infty  \sum_{m=0}^\infty f_nf_mt^{\alpha (n+m)}=0.
\end{equation}  
A straightforward computation gives 
\[ \pdc t^{\alpha n}=\beta_n t^{\alpha(n-1)},\]
where $\beta_n=\Gamma( \alpha n + 1)/ \Gamma(\alpha(n-1)+1)$. Replacing the previous identity in the equation \eqref{eq:f1}, we get
\begin{equation}\label{eq:series}
\sum_{n=1}^\infty f_n  \beta_n t^{\alpha (n-1)}+2 \sum_{n=0}^\infty  \sum_{m=0}^\infty f_nf_mt^{\alpha(n+m)}=0.
\end{equation}   
Collecting the terms of the same order and recalling that, by the initial condition in \eqref{eq:f}, 
$f_0=1$, we find
\[ 
\left\lbrace \begin{array}{lll}
\beta_1f_1+2 f_0^2=0&\quad\iff &f_1=-2\frac{\Gamma(1)}{\Gamma(\alpha+1)},\\[4pt]
\beta_2f_2+2 (f_0 f_1 + f_1 f_0)=0&\quad\iff &f_2=8\frac{\Gamma(1)}{\Gamma(2\alpha+1)},\\[4pt]
\beta_3f_3+2 (f_0f_2+f_1^2+f_2f_0)=0&\quad\iff &f_3=-2(f_1^2+2f_2)\frac{\Gamma(2\alpha+1)}{\Gamma(3\alpha+1)}, \\[4pt]
\vdots&&\vdots\\
\beta_nf_n+2\sum_{i+j=n-1}f_if_j=0&\quad\iff &f_n=-\frac{2}{\beta_n}\sum_{i+j=n-1}f_if_j.
\end{array}
\right.
 \]
From the previous relations, we can iteratively compute the coefficients of the power series $f(t)$ and we replace in \eqref{eq:explicit_sol}. Note that for $\alpha=1$, we get the power series of $\frac{1}{1+2t}$. As observed in \cite[Section 6.2.3]{po} where the problem \eqref{eq:f} for $\alpha=1/2$ is studied via power series expansion, finding the convergence interval of \eqref{eq:series} is not easy, but the series can be used for small values of $t$.
Indeed we numerically find that, for each $\alpha \in (0,1]$, there is a critical time $T_\alpha$ for which the power series $f(t)$ converges if $t \le T_\alpha$ and diverges if $t > T_\alpha$. The dependence of $T_\alpha$ on $\alpha$ is presented in Fig. \ref{fig:test1_critical_time}. Hence we use the representation formula \eqref{eq:explicit_sol}, with $f$ given by \eqref{eq:f}, for the exact solution of \eqref{eq:HJ_ex1} only for $t<T_\alpha$. For $t\ge T_\alpha$,  the viscosity solution of \eqref{eq:HJ_ex1}, which is defined for any $t$, cannot be anymore represented by means of \eqref{eq:explicit_sol} and we do not know  if there is an alternative explicit formula. Alternatively, we can calculate $f$ by means of a numerical methods for \eqref{eq:f}, but we do not pursue this approach here.

\begin{figure}
	\centering
	\includegraphics[width=60mm]{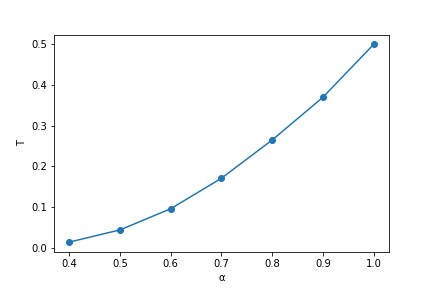}	
	\caption{Critical time for which the power series $f(t)$ converges}\label{fig:test1_critical_time}
\end{figure}

\begin{figure}
	\centering
	\includegraphics[width=60mm]{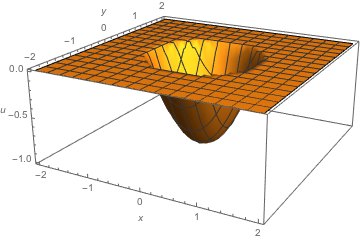}	
	\includegraphics[width=60mm]{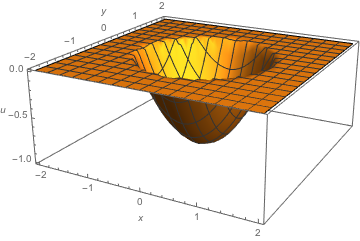}
	\caption{(A) Initial condition (B) Numerical solution at $t=0.2$}\label{fig:test1_sol_2d}
\end{figure}

\begin{figure}
	\centering
	\includegraphics[width=60mm]{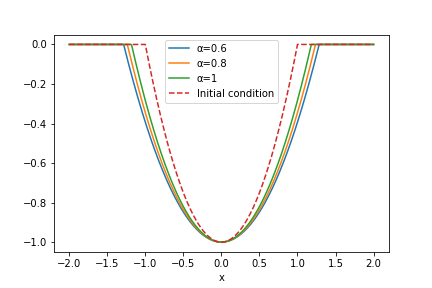}
	\includegraphics[width=60mm]{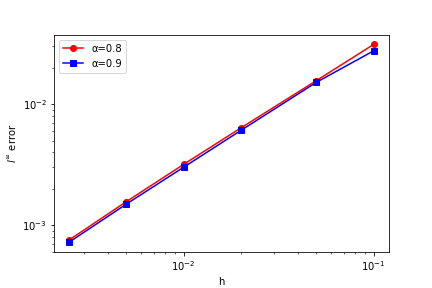}	
	\caption{(A) Numerical solutions at $t=0.2$ for different values of $\alpha$ (B) Convergence test for $t=0.2$}\label{fig:test1_different_alphas}
\end{figure}

The numerical solution at $t=0.2$ of \eqref{eq:HJ_ex1} where $\alpha = 0.8$ and $d=2$ computed by the Lax-Friedrichs scheme with $h = 10^{-1}$, $\Delta t = 10^{-3}$ and $\theta=1-2^{-\alpha}$ is provided in Fig. \ref{fig:test1_sol_2d}. We plot numerical solutions at $t=0.2$ for different values of $\alpha$ in Fig. \ref{fig:test1_different_alphas} (A) for $d=1$. We observe the convergent behavior of the solutions as $\alpha \to 1$. These solutions eventually converge to the solution of the classical case.

For the convergence test, we use $l^{\infty}$ error defined by the maximum difference between the exact and numerical solutions over all nodes. From Fig. \ref{fig:test1_different_alphas} (B), we determine that the convergence for the Lax-Friedrichs  scheme under the CFL condition is linear.

%%%%%%%%%%%%%%%%%%%%%%%%%%%%%%%%% Test 2 %%%%%%%%%%%%%%%%%%%%%%%%%%%%%%%%%
\subsection{Test 2}
\begin{figure}
	\centering
	\includegraphics[width=60mm]{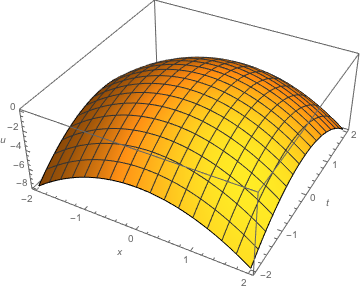}
	\includegraphics[width=60mm]{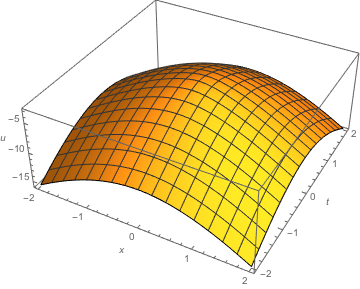}
	\caption{(A) Initial condition of \eqref{eq:HJ_ex2} for $\alpha = 0.8$ (B) Numerical solutions of \eqref{eq:HJ_ex2} when $\alpha = 0.8$ at $t=0.2$}\label{fig:test2_num_sol}
\end{figure}

We consider the Hamilton-Jacobi equation of the form
\begin{equation}\label{eq:HJ_ex2}
\left\{
\begin{array}{ll}
\pdc u+ |Du|=0, \qquad &(t,x)\in(0,\infty)\times\R^d,\\[6pt]
u_0(x)=-|x|^2, &x\in\R^d.
\end{array}
\right.
\end{equation}
For $\alpha=1$,   the unique viscosity solution of \eqref{eq:HJ_ex2} is  
\begin{equation*}
u (t,x) =  -(|x|+t)^2.
\end{equation*}
For $\alpha=1/2$, we look for a solution in the form
 \begin{equation}\label{eq:test2_sol}
u_{\frac 1 2} (t,x)=-|x|^2+\gamma t-2\beta|x|t^{1/2}.
 \end{equation} 
Computing the derivatives
\begin{align*}
&Du_{\frac 1 2}(t,x)=-2(|x|+\beta t^{1/2})\frac{x}{|x|}\\
&\pdh u_{\frac 1 2}(t,x)=-\frac{2\gamma}{\sqrt{\pi}}\sqrt{t}-\beta|x|\pi
\end{align*}
and replacing in the equation \eqref{eq:HJ_ex2}, we get
\[ \beta =\frac{2}{\sqrt{\pi}} , \quad \gamma =2,\]
and therefore 
\[u_{\frac 1 2} (t,x)=-|x|^2-2 t-\frac{4}{\sqrt{\pi}}|x|t^{1/2}.\]
We observe that the previous formula defines a viscosity solution of the problem. Indeed, for $x\neq 0$ and $t>0$, $u_{\frac 1 2}$ is differentiable and, since it is a solution in point-wise sense, it is also a viscosity solution  (see \cite[Prop. 2.10]{gn}). For $x=0$ and $t>0$, $u_{\frac 1 2}$ is not differentiable with respect to $x$, hence  the equation  has to be verified in viscosity sense. Since $u_{\frac 1 2}$ behaves as $-|x|$ near $x=0$, there cannot exist  a test function $\varphi$ such that $u_{\frac 1 2}-\varphi$ has a     minimum point at $(t,0)$, hence the supersolution condition is automatically satisfied. Moreover, for any test function $\varphi$ such that $u_{\frac 1 2}-\varphi$ has a maximum point at $(t,0)$, recalling  \eqref{eq:caputo_bis}, we see that $\varphi$ has to satisfy
\begin{align*}
	\pdh \varphi(t,0)=J[\varphi](t,0)+K_{(0,t)}[\varphi](t,0)\le J[u_{\frac 1 2}](t,0)+K_{(0,t)}[u_{\frac 1 2}](t,0)=\pdh u_{\frac 1 2}(t,0)=-\frac{4}{\sqrt{\pi}}t^{\frac 1 2},
\end{align*}
 and $|D\varphi(t,0)|\le 4 \sqrt{\pi}/\sqrt{t}$. Hence, also the subsolution condition at $(t,0)$ is satisfied.
For $\alpha\in (0,1)$, a similar computation gives that the solution of \eqref{eq:HJ_ex2}
is given by
\[ u_\alpha(t,x)=-|x|^2-
\frac{1}{\alpha\Gamma(2\alpha)}t^{2\alpha}
- \frac{2}{\alpha \Gamma(\alpha)}t^\alpha|x|. \]
Fig. \ref{fig:test2_num_sol} depicts the initial condition and the numerical solution at $t=0.2$ for $\alpha = 0.8$ obtained using the Lax-Friedrichs scheme in 2 dimensions. The parameter $\theta$ is always chosen equal to $1-2^{-\alpha}$. The numerical solutions corresponding to different values of $\alpha$ are plotted in Fig.  \ref{fig:test2_lax_fr_num_sol_different_alphas} (A) for $d=1$. We can see the same convergent behavior of the solutions as $\alpha \to 1$ as in the previous part. Moreover, from the convergence test in Fig. \ref{fig:test2_lax_fr_num_sol_different_alphas} (B), we observe that the convergence to the exact solution is linear. 

\begin{figure}
	\centering
	(A)\includegraphics[width=60mm]{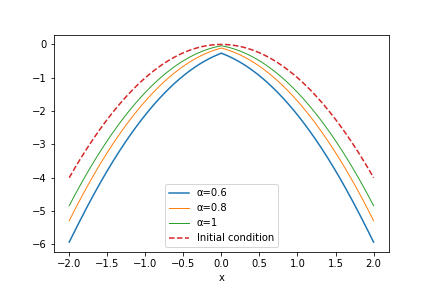}	
	(B)\includegraphics[width=60mm]{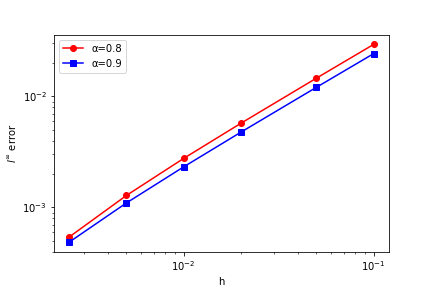}
	\caption{Lax-Friedrichs scheme: (A) Numerical solutions of \eqref{eq:HJ_ex2} at $T=0.2$ for different values of $\alpha$ (B) Convergence test for the fixed $\Delta t = 10^{-3}$} \label{fig:test2_lax_fr_num_sol_different_alphas}
\end{figure}

%%%%%%%%%%%%%%%%%%%%%%%%%%%%%%%%%%%%%%%%%%%%%%%%%%%%%%%%%%%%%%%%%%%%%%%%

% --------------------------------------------------------------
%     You don't have to mess with anything below this line.
% --------------------------------------------------------------
 
\end{document}